\documentclass[11pt, a4paper]{article}
\usepackage{amsmath,amssymb,amsthm,epsfig,setspace}
\usepackage{color}

\newtheorem{theorem}{Theorem}[section]
\newtheorem{Rmk}[theorem]{Remark}

\newtheorem{Thm}[theorem]{Theorem}
\newtheorem{Cor}[theorem]{Corollary}
\newtheorem{Prop}[theorem]{Proposition}
\newtheorem{Def}[theorem]{Definition}

\title{On spatial Gevrey regularity for some strongly dissipative second order evolution equations}
\author{Alain Haraux$^1$ and Mitsuharu  \^{O}tani$^2$\footnote{Partly supported by the Grant-in-Aid for Scientific Research \#15K13451, , the Ministry of Education, Culture, Sports, Science, and Technology, Japan}
\\ \\ 
$^1$ {\small Sorbonne Universit\'e, Universit\'e Paris-Diderot SPC, CNRS, INRIA,} \\
\noindent{\small Laboratoire Jacques-Louis Lions,  LJLL, F-75005,
Paris, France.} \\
\noindent $^2$ {\small Department of Applied Physics, School of Science and Engineering,} \\
{\small Waseda University 3-4-1, Okubo, Shinjuku-ku, Tokyo, Japan 169-855} \\ 
}
\date{}

\begin{document}
\maketitle
 \begin{abstract}  Let $A$ be a positive self-adjoint linear operator acting on a real 
  Hilbert space $H$ and $\alpha , c$ be positive constants.  We show that all solutions of the evolution equation $ u''+ Au+ c A^\alpha u' = 0$ 
 with $u(0) \in D(A^{\frac{1}{2}}), \ u'(0) \in H$ belong for all $t>0$ to the Gevrey space 
  $ G(A, \sigma)$  with  $\sigma = \min\{ \frac {1}{\alpha}, \frac {1}{1-\alpha}\}$. This result is optimal in the sense that $\sigma$ can not be reduced in general. 
For the damped wave equation (SDW)$_\alpha$ corresponding to the case where $A = -\Delta$ with domain $D(A) = \{ w\in H^1_0(\Omega), \Delta  w \in L^2(\Omega)\}$ with $\Omega$  any open subset of $\mathbb{R}^N$ and  $(u(0), u'(0)) \in H^1_0(\Omega)\times L^2(\Omega)$, the unique solution $u$ of
 (SDW)$_\alpha$ satisfies $ \forall t>0, \quad u(t) \in G^s(\Omega)$  with $ s = \min \{ \frac {1}{2\alpha}, \frac {1}{2(1-\alpha)}\} $, and  this result is also optimal. 
  
  \vspace{4ex}

\noindent{\bf Mathematics Subject Classification 2010 (MSC2010):}
35L10, 35B65, 47A60.

\vspace{4ex}

\end{abstract}

\newpage

\section{Introduction}
  Let $A$ be a positive self-adjoint linear operator acting on a real 
  Hilbert space $H$, let $A^\alpha$ be the fractional power of $A$ 
  of order $\alpha > 0$ and let $c$ be a positive constant. We consider the following evolution equation :
\begin{equation}\label{main}
 u''+ Au+ c A^\alpha u' = 0.
\end{equation}

The time regularity and smoothing effect at $t>0$ on solutions of \eqref{main} have been studied by quite a few authors, cf., e.g., \cite{CT1, CT2, CT3, G^2H}. In \cite{HO}, when $A$ is coercive, the authors of the present paper established by rather elementary means (avoiding complex analysis) that the semi-group generated by \eqref{main} is analytic if $\alpha \ge 1/2 $ and Gevrey of order $\frac{1}{2\alpha}$ if $\alpha < 1/2 $. But concerning spatial regularity, Theorem 5.1 from \cite{HO} only implies $C^\infty $ interior spatial regularity for $A = -\Delta$  or more generally for $A$ an elliptic operator with smooth coefficients. Now we shall study more  specifically the problem of regularity of $u(t)$ for $t>0$ when the initial state $(u(0), u'(0))$ lies in the standard energy space $V\times H$ with $V = D(A^{1/2}).$  Before stating any precise result, an important remark will allow us to understand that the exponent $\alpha =1/2 $, corresponding to the so-called structural damping (cf. \cite{CR}), is very special.  Indeed we notice that the time scaling $u(t): = v(kt)$ transforms the equation into $$  v''+ Bv + ck^{2\alpha-1} B^\alpha v' = 0$$  with $B= k^{-2}A$. Therefore whenever $\alpha \not=1/2 $, we can select $k$ in such a way that the coefficient of $B^{\alpha}v'$ becomes 1. On the contrary, if  $\alpha =1/2 $, the coefficient of $B ^ {1/2}v'$ is equal to $c$. The equations with different values of $c$ are all different, and they indeed have different properties even if $H = \mathbb{R}$ and $A= I$. In that most elementary case, the value $c=2$ is the threshold deciding the oscillatory or non-oscillatory character of solutions. Finally, if $\alpha =1/2 $ and $A$ has compact inverse, it can be seen that solutions of  \eqref{main}  of the form $$u(t)= w(t)\varphi $$ with $A\varphi= \lambda \varphi $ are all given by  $w(t) = z(\lambda t) $ where $z$ is a solution of the ODE 
$$  z''+z+cz' = 0. $$ In simple terms, those solutions all have the same shape up to time scaling, and a larger eigenvalue gives  rise to a ``faster" solution. 

As for the level of spatial smoothing effect, one might have thought that it increases with $\alpha$. But Remark 5.2 from \cite{HO} completely disqualifies this idea, since for $\alpha \ge 1$ there is no spatial smoothing effect at all. As we shall see the regularity of solutions for $t>0$ culminates for $\alpha =1/2 $. In the case of the wave equation (i.e. $A = -\Delta$), the value $\alpha =1/2 $ is the only one for which all solutions with initial data in the energy space are analytic in space for all $t>0$. For any $\alpha\in (0,1)$ other than $1/2$, the spatial smoothing effect for the wave equation is best described by a local Gevrey regularity in the sense of \cite{G}. 

Gevrey spaces have become rather popular when dealing with hyperbolic problems, and after the pioneering works on analyticity of solutions to PDE such as \cite {K1, K2} based on the method of \cite{Hor}, there are more recent papers dealing with interior Gevrey regularity of solutions to elliptic equations, cf.,  e.g., \cite{ N-Z, T}. Theorem 3 from \cite{ N-Z} will allow us to study rather easily the level of spatial smoothing effect at $t>0$ for  equation  \eqref{main} in the concrete PDE cases. 

The plan of the paper is as follows: In Section 2 we state and prove the main abstract result, giving first a complete explicit proof in the coercive case and showing the additional necessary steps to obtain the result in the general, possibly non-coercive, case. In Section 3 we apply the general result to the case of the wave equation in any domain of $\mathbb{R}^N$, using a useful result from \cite{ N-Z} which connects local ultra-differentiability properties of solutions to Gevrey type estimates for powers of the main operator $A$. In Section 4 we establish several optimality results, both for the abstract theorem and for the local Gevrey class of solutions to the 1D wave equation with strong dissipation. Section 5 is devoted to higher order equations in space and various extensions of the results. Finally the appendix discusses equivalent definitions of the Gevrey classes for both operators and functions, and develops some useful tools used in the previous sections. 


\section{An abstract regularity result} Before stating our main result we need to introduce some notation. Inspired by the notion of ``analytic vectors" for an operator $A$ defined by Nelson in \cite{N} and the Gevrey regularity class of functions (cf. \cite{G}) we define Gevrey vectors as follows 

\begin{Def} Let $A$ be any positive self-adjoint operator on $H$. A vector $u\in H$ will be called Gevrey of order $s>0$ with respect to $A$ if $u\in D(A^n) $ for all n and for some $R>0$ we have $$ \forall n\in\mathbb{N}-\{0\}, \quad  |A^n u|\le R^n n^{s n}.$$ 
In this case we write $ u\in G(A, s).$ 
\end{Def}

\begin{Rmk} The apparent divergence between this definition and those of \cite{N} and \cite{G} will be clarified in the appendix, Proposition 6.1. 
\end{Rmk}

\begin{Rmk} It is clear that for any $ s>0$,$$  G( I+ A, s) = G(A, s) $$   and for any $\lambda>0, s>0$,  $$  G(\lambda A, s) = G(A, s).$$  
\end{Rmk}

\begin{Rmk} We shall prove in the appendix that for any positive self-adjoint operator $A$ and any positive numbers $\alpha, s$ we have 
$$  G (A, s)  = G(A^\alpha, \alpha s).$$  
\end{Rmk}
\bigskip\noindent Our main result is the following. 

\begin{Thm}\label{mainth} For any $\alpha\in (0,1)$ and any  $(u_0, u_1)) \in V\times H$ , the unique solution $u$ of \eqref{main} with initial date $u(0) = u_0, \quad u'(0) = u_1$ satisfies 
\begin{equation}\label{mainrs}
 \forall t>0, \quad u(t) \in G(A, \sigma ); 
\quad \sigma = \min\{ \frac {1}{\alpha}, \frac {1}{1-\alpha}\}. 
\end{equation}
\end{Thm}
 
 \begin{proof} We consider successively the cases $0< \alpha\le \frac{1}{2}$ and $\frac{1}{2}<\alpha<1. $ In both cases we shall use the notation 
 $$ E = V \times H; \ U = (u,u') $$ and we shall  occasionally drop $t$ to shorten some formulas. 
We first observe that for any positive self-adjoint operator $B$ and any sufficiently smooth solution u of 
$$ 
    u'' + Au + Bu' = 0, 
$$  
we have the formal identity $$ \frac{d}{dt} \{|u'|^2 + |A^{\frac{1}{2}}u|^2 + {\frac{1}{2}}|Bu|^2 +(Bu, u')\} = - \{|B^{\frac{1}{2}}u'|^2 + (Au, Bu)\} ).$$

On the other hand $$ \Phi(u, u'):= |u'|^2 + |A^{\frac{1}{2}}u|^2 + {\frac{1}{2}}|Bu|^2 +(Bu, u')=  {\frac{1}{2}}|u'|^2 + |A^{\frac{1}{2}}u|^2 +{\frac{1}{2}}|u'+Bu|^2 $$ yielding the convenient inequalities $$ 0\le {\frac{1}{2}}|u'|^2 + |A^{\frac{1}{2}}u|^2 \le \Phi(u, u') \le \frac{3}{2}|u'|^2 + |A^{\frac{1}{2}}u|^2 + |Bu|^2.
$$ 
In particular in this fairly general context we have the (formal) inequality 
\begin{equation}\label{est:Bhalfuprime:above}
\int_0^t \{|B^{\frac{1}{2}}u'(s)|^2 + (Au(s), Bu(s))\} ds \le \frac{3}{2}|u'(0)|^2 + |A^{\frac{1}{2}}u(0)|^2 + |Bu(0)|^2. 
\end{equation}

In order to make the proof easier to follow, we give first a complete proof when $A$ is coercive. Let us first apply this formula for $ B = cA\alpha$ with $0< \alpha \le \frac{1}{2}$.  We obtain, since B commutes with all powers of $A$: 
 $$  \min \{c, c^2\}  ~\! t | A^{\frac{\alpha}{2}} u(t), A^{\frac{\alpha}{2}} u'(t)|^{2}_ {V\times H}\le \int_0^t \{|B^{\frac{1}{2}}u'(s)|^2 + (Au(s), Bu(s))\} ds \le K E_0^2$$
 with $ E_0 := |(u(0), u'(0))|_{V\times H}.$ Thus by iterating the process twice and replacing $t$ bu $t/2$ we obtain 
$$ | A^{\alpha} u(t), A^{\alpha} u'(t)|_ {V\times H} \le  K' t^{-1}E_0. $$ 
Replacing $t$ by $\frac{t}{m}$ and iterating $m$ times, we obtain with $K'= K^{\alpha}$
$$ | A^{m\alpha} u(t), A^{m\alpha} u'(t)|_ {V\times H}\le [ \frac {K' m} {t }]^{ m} E_0. $$ 
Hence both $u(t)$ and $u'(t)$  belong to $G(A^\alpha, 1) = G(A, \frac{1}{\alpha})$ for all $t>0$.  
\\ 
For $ B = cA\alpha$ with $\frac{1}{2}\le \alpha <1$, we find by 
\eqref{est:Bhalfuprime:above}  
 $$ t | A^{\frac{\alpha}{2}} u(t), A^{\frac{\alpha}{2}} u'(t)|^{2}_ {V\times H}\le  K | A^{\alpha- \frac{1}{2}} u(0), A^{\alpha- \frac{1}{2}} u'(0)|^{2}_ {V\times H}.$$ Thus by setting $v(t)= A^{\alpha- \frac{1}{2}} u(t) $, we obtain 
$$ 
| A^{1-\alpha} v(t), A^{1-\alpha} v'(t)|_ {V\times H} \le  K' t^{-1}|(v(0), v'(0))|_{V\times H}. 
$$ 
 Then by iteration as before: we find that both $v(t)$ and $v'(t)$  belong to $G(A^{1-\alpha}, 1) = G(A, \frac{1}{1-\alpha})$ for all $t>0$, whence follows \eqref{mainrs}. 
\bigskip 
 
Let us now consider the general case, but only with $\alpha\not = \frac{1}{2}$ and then, using the remark of the introduction, we can drop the constant $c.$ 
 
 1) The case $0< \alpha < \frac{1}{2}, c=1$. We start from the inequality 
 
  $$ t | A^{\frac{\alpha}{2}+ \frac{1}{2}} u(t), A^{\frac{\alpha}{2}} u'(t)|^{2}_ {H\times H}\le  \frac{3}{2}|u'(0)|^2 + |A^{\frac{1}{2}}u(0)|^2 + |A^{\alpha} u(0)|^2.$$ When $A$ is non-coercive, the term $|A^{\alpha} u(0)|^2$ cannot be controlled by $|A^{\frac{1}{2}}u(0)|^2$  only. Instead we may use 
$$ 
|A^{\alpha} u(0)|^2\le |u_0|^2+ |A^{\frac{1}{2}}u(0)|^2 
,$$ 
which implies
   $$  
\frac{3}{2}|u'(0)|^2 + |A^{\frac{1}{2}}u(0)|^2 + |A^{\alpha} u(0)|^2 \le 2[|u'(0)|^2 + ||u(0)||^2], 
$$ 
where the norm $|| . ||$ is defined  on $V: = D(A^{1/2}$ by 
 $$ 
 \forall x\in V, \quad  ||x|| = ( |x|^2 +  |A^{\frac{1}{2}} x|^2)^{1/2}.
$$  This will be the only norm used on $V$ later on and we set with $ V\times H: = E$  $$ \forall U = (u, v) \in E, \quad |U|_E = (||u||^2+|v|^2)^{1/2}.$$ 
 A difference with the energy norm used in the coercive case is that now the modified norm is no longer non-increasing. Instead it is easy to prove that for any solution $U = (u, u') $ of \eqref{main}, we have $$\forall t\ge 0, \quad  |U(t)|_E^2 \le e^t  |U(0)|_E^2. $$ 
Next we have to handle powers of the operator $I+A$ instead of $A$. More precisely we want to estimate 
   $ | ((I+A)^{\frac{\alpha}{2}} u(t), (I+A)^{\frac{\alpha}{2}} u'(t))|^{2}_ {E}$ to compare it with $ | (u(0), u'(0))|^{2}_ {E}$. This is not really difficult but has to be done carefully. As a first step we observe that 
   
$$|(I+A)^{\frac{\alpha}{2}} u'(t))|^{2} \le |u'(t)|^2 +  |A^{\frac{\alpha}{2}} u'(t)|^2 $$ as a consequence of the operator inequality $$ (I+A)^\alpha \le I+ A^{\alpha}$$ is valid for any $\alpha \in [0, 1]$. Then we note  
 $$ 
     || (I+A)^{\frac{\alpha}{2}} u(t)||^2   \leq   
              | (I+A)^{\frac{\alpha}{2}} u(t)|^2 + | (I+A)^{\frac{\alpha+1}{2}} u(t)|^2. 
$$ 
First by contraction we have 
$$  
| (I+A)^{\frac{\alpha}{2}} u(t)|^2 \le | (I+A)^{\frac{1}{2}} u(t)|^2 = ||u(t)||^2.
$$ 
Then 
$$  | (I+A)^{\frac{\alpha+1}{2}} u(t)|^2\le |u(t)|^2 +  | A^{\frac{\alpha +1}{2}} u(t)|^2 \le ||u(t)||^2 + | A^{\frac{\alpha +1}{2}} u(t)|^2$$ so that we obtain 
$$ || (I+A)^{\frac{\alpha}{2}} u(t)||^2 \le 2 ||u(t)||^2 + | A^{\frac{\alpha +1}{2}} u(t)|^2$$ and finally 
$$| ((I+A)^{\frac{\alpha}{2}} u(t), (I+A)^{\frac{\alpha}{2}} u'(t))|^{2}_ {E} \le  | A^{\frac{\alpha +1}{2}} u(t)|^2 + |A^{\frac{\alpha}{2}} u'(t)|^2 + 2 |U(t)|^2_E.$$ 
Because  $|U(t)|_E^2 \le e^t  |U(0)|_E^2$ we now obtain 
$$ 
 | ((I+A)^{\frac{\alpha}{2}} u(t), (I+A)^{\frac{\alpha}{2}} u'(t))|^{2}_ {E} 
      \le  2 (e^t + \frac{1}{t}) |U(0)|_E^2.
 $$ 
This is enough to conclude in a few easy steps. 
\\
 
 2) The case $\frac{1}{2}< \alpha < 1, c=1$. We start from the inequality 
 
  $$ 
    t | A^{\frac{\alpha}{2}+ \frac{1}{2}} u(t), A^{\frac{\alpha}{2}} u'(t)|^{2}_ {H\times H}\le  \frac{3}{2}|u'(0)|^2 + |A^{\frac{1}{2}}u(0)|^2 + |A^{\alpha} u(0)|^2.
$$ 
First we have  
  $$ 
    |A^{\frac{1}{2}}u(0)|^2   \le |u_0|^2+ |A^{\frac{1}{2}} u(0)|^2 
           = |(I+A) ^{\frac{1}{2}}u(0)|^2 
             \le  |(I+A) ^{\alpha}u(0)|^2.
$$ 
Moreover 
$$  
   |A^{\alpha} u(0)|^2 \le  |(I+A)^{\alpha} u(0)|^2;\quad |u'(0)|^2 \le  |(I+A) ^{\alpha-\frac{1}{2}}u'(0)|^2
$$ 
and we obtain 
$$ 
  \frac{3}{2}|u'(0)|^2 + |A^{\frac{1}{2}}u(0)|^2 + |A^{\alpha} u(0)|^2 
         \le 2[|(I+A) ^{\alpha-\frac{1}{2}}u'(0)|^2 + ||(I+A) ^{ \alpha}u(0)||^2].
$$ 
So the RHS of the basic inequality is bounded by
$$ 
   2 |(I+A) ^{\alpha-\frac{1}{2}}U(0)|^2_E.
$$ 
   It remains to bound the quantity $ | ((I+A)^{\frac{\alpha}{2}} u(t), (I+A)^{\frac{\alpha}{2}} u'(t))|^{2}_ {E}.$ We first write 
   
   $$ |(I+A)^{\frac{\alpha}{2}} u'(t))|^{2} =  |(I+A)^{\frac{1-\alpha}{2}} (I+A)^{\alpha -\frac{1}{2}}u'(t))|^{2} $$ 
   $$  \le  |(I+A)^{\alpha - \frac{1}{2}} u'(t))|^{2} +  |A^{\frac{1-\alpha}{2}}(I+A)^{\alpha - \frac{1}{2}}  u'(t))|^{2}$$
   $$\le   |(I+A)^{\alpha - \frac{1}{2}} u'(t))|^{2} +  |A^{\frac{1-\alpha}{2}} u'(t))|^{2} +  |A^{\frac{\alpha}{2}} u'(t))|^{2}$$
   $$ \le  |(I+A)^{\alpha - \frac{1}{2}} u'(t))|^{2} +  |u'(t))|^{2} +  2|A^{\frac{\alpha}{2}} u'(t))|^2 $$ 
 $$ 
\le 2 [ |(I+A)^{\alpha - \frac{1}{2}} u'(t))|^{2}  +  |A^{\frac{\alpha}{2}} u'(t))|^2 ].
 $$
   
   A quite similar calculation gives $$ |(I+A)^{\frac{\alpha+1}{2}} u(t))|^{2}  \le 2 [ |(I+A)^{\alpha } u(t))|^{2}  +  |A^{\frac{\alpha +1}{2}} u(t))|^2 ] $$ and by addition we obtain 
   $$  |(I+A)^{\frac{\alpha}{2}} u'(t))|^{2} + |(I+A)^{\frac{\alpha+1}{2}} u(t))|^{2}  \le |A^{\frac{\alpha}{2}} u'(t))|^2 + |A^{\frac{\alpha +1}{2}} u(t))|^2 +2 e^t |(I+A) ^{\alpha-\frac{1}{2}}U(0)|^2_E .$$ Finally we obtain 
   $$  |(I+A)^{\frac{\alpha}{2}} u'(t))|^{2} + |(I+A)^{\frac{\alpha+1}{2}} u(t))|^{2}  \le 2( e^t  + \frac{1}{t}) |(I+A) ^{\alpha-\frac{1}{2}}U(0)|^2_E . $$
   
   Setting  $ v(t):= (I+A) ^{\alpha-\frac{1}{2}}u(t) $ we obtain 
   $$  |(I+A)^{\frac{1-\alpha}{2}} V(t))|^{2} _E \le 2( e^t  + \frac{1}{t}) |V(0)|^2_E . $$ Then the conclusion follows easily as previously. \\
   
   3) The case $\alpha =\frac{1}{2}$ In this case we can follow either the method of case 1 or case 2 but the constant $c$ will appear in the estimate. Since the calculation is just a variant and has been done completely in the coercive case, for the sake of brevity we skip the details. 
   
   \end{proof}

\section{The case of the wave equation with strong damping}

\begin{Def} Let $\Omega$ be any open subset of $\mathbb{R}^N$. A function $f: \Omega \rightarrow \mathbb{R}$ will be called Gevrey of order $s>0$ in $\Omega$ if $f\in C^\infty (\Omega ) $ and for any $K$ compact subset of $\Omega$, there is  $R= R(K)>0$ such that, for any differential monomial 
$ D^p: = D_1^{p_1} D_2^{p_2} ...D_N^{p_N}$ we have 
$$ ||D^p f||_{L^{\infty} (K)}\le R ^{|p|} |p|^{s|p|}. $$

In this case we write $ f\in G^s(\Omega).$ 
\end{Def}

\begin{Rmk} This definition is slightly different from the definition given in the historical literature, in particular in the seminal paper \cite{G}, but it is in fact equivalent and more condensed, cf. appendix, Proposition 6.1. \end{Rmk}

\begin{Thm}\label{wave} Let $\Omega$ be any open subset of $\mathbb{R}^N$ and $H= L^2(\Omega), V = H^1_0(\Omega)$.  
Let $A = -\Delta$ with domain $D(A) = \{ w\in V, \Delta v\in H\}$ For any $\alpha\in (0,1)$ and any  $(u_0, u_1) \in V\times H$ , the unique solution $u$ of \eqref{main} with initial data $u(0) = u_0, \,\,u'(0) = u_1$ satisfies 
\begin{equation}\label{mainrs2}
 \forall t>0, \quad u(t) \in G^s(\Omega) \end{equation} with\begin{equation}s = \min \{ \frac {1}{2\alpha}, \frac {1}{2(1-\alpha)}\} . \end{equation} In particular for $\alpha = 1/2$, $u(t)$ is analytic inside  $\Omega$ for all $t>0$. 
 \end{Thm}
 
\begin{proof}  By Theorem 3 of \cite{N-Z}, when $A $ is a second order elliptic operator, we have $ G(A, 2s)\subset G^s(\Omega)$. Therefore the result is an immediate consequence of Theorem \ref{mainth} .\end{proof}

\section{Optimality results}

The next two results provide a very strong optimality statement for Theorem \ref{mainth}.  

\begin{Thm}\label{opt} Let $A$ be coercive with $A^{-1}$ compact and assume the two following conditions 

i) For some $\varepsilon>0, \delta>0$ we have $$\forall n\ge 1, \quad \lambda_n\ge \delta n^\varepsilon. $$ 

ii) For some $C>1$  we have $$\forall n\ge 1, \quad \lambda_{n+1} \le C\lambda_n. $$ 
Assume $\alpha \in (1/2, 1).$ Then there is a solution of  \eqref{main} with initial data $u(0) = u_0, u'(0) = u_1$ in $V\times H$ for which we have 
$$ \forall   k\ge 1,  \forall t>0, \quad  |A^k u(t)| \ge  [\delta(t)k]^ {\frac{k}{1-\alpha}} $$
where $\delta(t)>0$ for all $t>0.$ 

\end{Thm}

\begin{proof}  Assuming  ${\lambda_{n_0}}^{2\alpha-1}> 4$ we look for a solution of \eqref{main} of the form $$ u(t) = \sum_{n\ge n_0} c_n e^{-\mu_n t} \varphi_n $$ with $\varphi_n$ a sequence of normalized eigenvectors corresponding to the eigenvalues $\lambda_n$ and $$ \mu_n: = \frac{\lambda_n^\alpha}{2}- \sqrt{\frac{\lambda_n^{2\alpha}}{4}- \lambda_n}= \frac {\lambda_n}{\frac{\lambda_n^\alpha}{2}+ \sqrt{\frac{\lambda_n^{2\alpha}}{4}- \lambda_n}}. $$ 
We observe that for all $n$ $$\lambda_n^{1-\alpha} \le \mu_n \le 2\lambda_n^{1-\alpha}. $$ 
It is clear that $u$ is indeed  a solution of \eqref{main} if the coefficients $c_n$ tend to $0$ fast enough when $n$ grows to infinity. A sufficient condition for that is $$c_n = \lambda_n^{-K};\quad K> 1 +\frac{1}{2\varepsilon}.$$ 
Now for all k  we have $$A^k u(t) = \sum_{n\ge n_0} c_n\lambda_n^k e^{-\mu_n t} \varphi_n  $$  and as a consequence of orthogonality of the eigenvectors in $H$ we obtain 
$$ \forall   n\ge n_0,  \forall t>0, \quad  |A^k u(t)| \ge c_n \lambda_n^k e^{-\mu_n t}.$$ 
In that inequality we choose 
$$ n= \inf \{m\ge n_0, \lambda_m \ge k^{\frac{1}{1-\alpha}}\}. $$ 
Then for $k$ large enough we must have $n>n_0$. In this case $ \lambda_{n-1}< k^{\frac{1}{1-\alpha}}$ and then $ \lambda_n \le Ck^{\frac{1}{1-\alpha}}$. It follows that  $$ c_n \lambda_n^k e^{-\mu_n t} \ge \lambda_n^{-K} k^{\frac{k}{1-\alpha}} e^{-2C^{1-\alpha} kt }\ge C^{-K}k^{-\frac{K}{1-\alpha}}e^{-2C^{1-\alpha} kt }k^{\frac{k}{1-\alpha}}.$$ This concludes the proof. 
\end{proof} 

\begin{Thm}\label{opt2} Let $A$ be as in the statement of Theorem \ref{opt}. Assume $\alpha \in (0, 1/2).$ Then there is a solution of  \eqref{main} with initial data $u(0) = u_0, u'(0) = u_1$ in $V\times H$ for which we have 
$$ \forall   k\ge n_0,  \forall t>0, \forall \theta>t, \quad  \int _t^{\theta} |A^k u(s)|^2ds  \ge  [\delta(t, s)k]^ {\frac{2k}{\alpha}}, $$
where $\delta(t,s)>0$. 

\end{Thm}

\begin{proof}  Assuming ${\lambda_{n_0}}^{1-2\alpha}> 1/4$, we shall find find a solution of \eqref{main} of the form $$ u(t) = \sum_{n\ge n_0} c_n e^{-\frac{\lambda_n^\alpha}{2}t} \cos \left(t \sqrt{\lambda_n- \frac{\lambda_n^{2\alpha}}{4}}\right)\varphi_n $$ with $\varphi_n$ a sequence of normalized eigenvectors corresponding to the eigenvalues $\lambda_n$.  
It is clear that $u$ is indeed  a solution of \eqref{main} if the coefficients $c_n$ tend to $0$ fast enough when $n$ grows to infinity. A sufficient condition for that is $$c_n = \lambda_n^{-K};\quad K> 1 +\frac{1}{2\varepsilon}.$$ 
Now for all k  we have $$A^k u(t) = \sum_{n\ge n_0} c_n \lambda_n^k e^{-\frac{\lambda_n^\alpha}{2}t} \cos \left(t \sqrt{\lambda_n- \frac{\lambda_n^{2\alpha}}{4}}\right)\varphi_n   $$  and as a consequence of orthogonality of the eigenvectors in $H$ we obtain 
$$ \forall   n\ge n_0,  \forall t>0, \quad  |A^k u(t)|^2 \ge c_n^2 \lambda_n^{2k} e^{-\lambda_n^\alpha  t} \cos^2 \left(t \sqrt{\lambda_n- \frac{\lambda_n^{2\alpha}}{4}}\right).$$ 
In that inequality we choose 
$$ n= \inf \{m\ge n_0, \lambda_m \ge k^{\frac{1}{\alpha}}\}. $$ 
Then for $k$ large enough we must have $n>n_0$. In this case $ \lambda_{n-1}< k^{\frac{1}{\alpha}}$ and then $ \lambda_n \le Ck^{\frac{1}{\alpha}}$. The end of the proof is now quite similar to the proof of the previous result, the only difference being integration in t to handle the oscillating term and the remark that  the integral of the function $\cos^2 \left(t \sqrt{\lambda_n- \frac{\lambda_n^{2\alpha}}{4}}\right)$ on any time interval $J$ tends to $|J| /2$ as $n$ tends to infinity. We skip the details.
\end{proof} 

\begin{Rmk}\label{opt3} Let $A$ be as in the statement of Theorem \ref{opt}. Assume $\alpha =1/2.$ Then there is a solution of  \eqref{main} with initial data $u(0) = u_0, u'(0) = u_1$ in $V\times H$ for which we have 
$$ \forall   k\ge n_0,  \forall t>0, \forall \theta>t, \quad  \int _t^{\theta} |A^k u(s)|^2ds  \ge  [\delta(t, s)k]^ {4k}, $$
where $\delta(t,s)>0$. The proof follows the line of proof of either Theorem \ref{opt} if $c\ge 2$ or Theorem \ref{opt2} if $c<2$ . We skip the details.

\end{Rmk}

\begin{Cor}\label{opt3} Let $\Omega$ be a bounded interval of $\mathbb{R}$ and $H= L^2(\Omega); V = H^1_0(\Omega).$ Then for any $r<s=\min \{ \frac {1}{2\alpha}, \frac {1}{2(1-\alpha)}\}$ given by Theorem \ref{wave}, there is a solution of  \eqref{main} with initial data $u(0) = u_0, u'(0) = u_1$ in $V\times H$ for which $u(t)$ never belongs to $G^r(\Omega)$. In particular the solutions are not analytic in general for $t>0$ if $\alpha\not= 1/2$.\end{Cor}

\begin{proof} By a translation and a space-scaling we can reduce the question to the case $ \Omega = (0, 3\pi)$ and show that for some solutions, the Gevrey estimates in the interior subset  $ \omega = (0, \pi)$  are not better than those of the general theorems. The eigenfunctions of the Dirichlet-Laplacian in  $ \Omega = (0, 3\pi)$ are the functions $\sin \frac{kx}{3}$. We choose the solutions with initial data spanned by the functions  $\sin (mx) $ only, restricting ourselves to $k = 3m$. These solutions satisfy the same equation in  $ \omega = (0, \pi)$  with homogeneous Dirichlet boundary conditions, and here the estimates of successive derivatives correspond exactly to the double exponents for the same powers of the Laplacian. Therefore the examples constructed in The two previous theorems and the remark provide solutions having in $ \omega = (0, \pi)$ for $t>0$ the exact Gevrey regularity allowed by Theorem \ref{wave}, and not more.\end{proof}

\section{Other examples and possible extensions}

The general Theorems apply also to plate (beam in 1D) equations, either  clamped or simply supported. 

\begin{Thm}\label{CPlate} Let $\Omega$ be any open subset of $\mathbb{R}^N$ and $H= L^2(\Omega); V = H^2_0(\Omega)$ 
Let $A = -\Delta$ with domain $D(A) = \{ w\in V, \Delta v\in H\}$ For any $\alpha\in (0,1)$ and any  $(u_0, u_1) \in V\times H$, the unique solution $u$ of \eqref{main} with initial data $u(0) = u_0, \quad u'(0) = u_1$ satisfies 
\begin{equation}\label{mainrs2}
 \forall t>0, \quad u(t) \in G^s(\Omega) \end{equation} with \begin {equation} \label{mainrs3}
 s = \min \{ \frac {1}{4\alpha}, \frac {1}{4(1-\alpha)}\} . \end{equation}
 In particular for all $\alpha \in [1/4, 3/4] $, $u(t)$ is analytic inside  $\Omega$ for all $t>0$. 
 \end{Thm}
 
 \begin{proof}  By Theorem 3 of \cite{N-Z}, when $A $ is a fourth order elliptic operator, we have $ G(A, 4s)\subset G^s(\Omega)$. Therefore the result is an immediate consequence of Theorem \ref{mainth} .\end{proof} 
 
 \begin{Thm}\label{SSPlate} Let $\Omega$ be any open subset of $\mathbb{R}^N$ and $H= L^2(\Omega); V = H^2\cap H^1_0(\Omega)$ 
Let $A = -\Delta$ with domain $D(A) = \{ w\in V, \Delta v\in H\}$ For any $\alpha\in (0,1)$ and any  $(u_0, u_1) \in V\times H$, the unique solution $u$ of \eqref{main} with initial data $u(0) = u_0, \quad u'(0) = u_1$ satisfies \eqref{mainrs2} and \eqref{mainrs3}.
 In particular for all $\alpha \in [1/4, 3/4] $, $u(t)$ is analytic inside  $\Omega$ for all $t>0$. 
 \end{Thm}
 
 \begin{proof}  Same as for Theorem \ref{CPlate}.\end{proof} 
 
 \begin{Rmk} It is clear from the structure of the proof of Theorem \ref{mainth} that $B$ does not need to be an exact power of $A$ for the result to hold true. For instance, a linear combination with positive coefficients of an arbitrary number of powers of $A$ : $B= \sum c_i A^{\alpha_i} $ will give the same regularity result with $\alpha$ replaced by the highest exponent. Similarly $ B = c (dI+ A)^\alpha $ or a sum of such operators will give the same result as $ A^\alpha$. \end{Rmk}  
\section{Appendix} In this appendix we establish a few properties of general interest that have been used in the proofs of our main results.  \subsection{Equivalent formulations of Gevrey spaces} In this section, we clarify once and for all the connection between the various definitions of Gevrey regularity found in the literature. The original definition by Gevrey in \cite{G} involves a power of the multi-factorial  $ p! : p_1!...p_N! $ when $p:= (p_1,...p_N)$ is an N-vector with integer coordinates. This was in fact motivated by the possibility of considering different regularity levels in the N different differentiation directions. When one is not interested in doing that, one might consider replacing the multi-factorial by the factorial of the total differentiation order $|p|: = p_1+...+ p_N$, i.e., consider $|p|!$ instead of $ p! $  Do we still find the same regularity class? Alternatively, many authors replaced in the definition the factorials $p_j !$ by $p_j^{p_j} .$ , justifying usually this change by Stirling's asymptotic formula. Then what about using simply the apparently larger number $|p|^{|p|}$? The next result shows that all those notions are equivalent, and we can quantify exactly the equivalence constants as a function of the dimension only. 

\begin{Prop}\end{Prop} For any $N\in \mathbb{N}^*$ and any $p= (p_1,...p_N)\in {\mathbb{N}^*}^N$ we have 
$$ p_1!...p_N!= p! \le p^p (= \prod _1^N p_i^{p_i}) \le |p|^{|p|} \le (4^{N-1})^{|p|}p^p \le (4^{N-1}e)^{|p|} p! $$

\begin{proof}The first 2 inequalities $ p_1!...p_N!= p! \le p^p (= \prod _1^N p_i^{p_i}) \le |p|^{|p|} $ are completely obvious. So we are left to check that $|p|^{|p|} \le (4^{N-1})^{|p|}p^p \le (4^{N-1}e)^{|p|} p! $
We do this in 3 steps. \\

Step 1. In the case of two components, we claim that $$ \forall (p,q)\in {\mathbb{N}^*}^2, \quad (p+q)^{p+q}\le 2^{2(p+q)} p^p q^q $$
Indeed assuming $p\leq$ with $2^r p\le q\le 2^{r+1} p$, we obtain first 
$$ q^p\le (2^{r+1} p)^p = p^p 2^{p+rp} \le p^p 2^{p+q}$$ since $rp\le 2^r p \le q.$ Then 
$$ (p+q)^{p+q} \le (2q)^{(p+q)}= 2^{p+q} q^{p+q}= q^q 2^{p+q} q^{p}\le q^q 2^{p+q} p^p 2^{p+q}$$ and the claim is justified. \\

Step 2. In the case of three or more components, we prove by induction that  $$|p|^{|p|} \le (4^{N-1})^{|p|} p^p $$

Assuming the result to be true for N-1 components, we use the result for 2 components with $p_1$ and $ p_2+...p_N$ in place of $p$ and $q$, which gives
$$ (p_1+p_2 +...p_N)^{(p_1+p_2 +...p_N)} \le 2^{2|p|} p_1^{p_1}(p_2 +...p_N)^{(p_2 +...p_N)} $$ 
$$ \le 2^{2|p|} p_1^{p_1} 2^{2(N-2)(p_2+...+p_N)}{p_2}^{p_2} ...{p_N}^{p_N} \le 2^{2(N-1)|p|} p_1^{p_1} {p_2}^{p_2} ...{p_N}^{p_N} $$  and the result follows. \\

Step 3. The concavity of the function $\ln$ on $(0, 1)$ implies that for any integer $k$, we have $k^k\le e^k k!$ Hence $$ p^p \le  e^{|p|}p! $$

The proof is concluded by combining Step 2 and Step 3. 

\end {proof}  
\subsection {Gevrey spaces for a power of an operator}  
\begin{Prop} For any positive self-adjoint operator $A$ and any positive numbers $\alpha, s$ we have 
$$  G (A, s)  = G(A^\alpha, \alpha s).$$  
\end{Prop} 
\begin {proof}  Assume that $u\in D(A^n) $ for all n and for some $R>0$ we have $$ \forall n\in\mathbb{N}, \quad  |A^n u|\le R^n n^{s n}.$$ 
Then for any $\theta \in (0, 1)$, by the following interpolation inequality 
\begin{equation*}
| A^\theta u | \leq  | A u|^{\theta} |u|^{1-\theta} \qquad \forall u \in D(A), 
\end{equation*}
we obtain 
 $$
    \forall n\in\mathbb{N}, \quad  
      |A^{n\theta}  u| \le  ~\! R^{n\theta} n^{\theta s n} |u| ^{(1-\theta) n} 
        =  ~\![R^{\theta} |u| ^{(1-\theta)}]^n n^{(\theta s) n} $$ which means exactly that $u \in G(A^\theta, \theta s).$ Hence for all $\alpha \in (0,1],$ 
$$ 
 G (A, s)  \subset  G(A^\alpha, \alpha s).
$$
 We also find  
\begin{align*} 
    \forall m \in \mathbb{N}, \quad  |A^{m+\theta}  u| 
            & \le R^m m^{sm} ~\! |Au|^{\theta} |u|^{(1-\theta)}
\\
            & \le   |u|^{1-\theta} ~\! 
                      [\max \{R, |Au|\}]^{m+\theta} (m +\theta)^{s(m +\theta)}, 
\end{align*} 
 which implies in particular
$$ 
     \forall \tau \ge 1, \quad  |A^{\tau} u| \le [\max \{R, |u|, |Au|, 1\}] ^{\tau} (\tau)^{s\tau}. 
$$ 
 In particular, taking $ \tau = n \alpha $, we have 
 $$ \forall n\in\mathbb{N}-\{0\}, \quad   |A^{n\alpha} u|\le [\max \{R, |u|, |Au|, 1\}\alpha^s] ^{n\alpha} n^{sn\alpha}$$ which means exactly that $u \in G(A^\alpha, \alpha s).$ Hence for all $\alpha \ge 1,$ $$  G (A, s)  \subset  G(A^\alpha, \alpha s)$$ Finally $$ \forall \alpha>0, \forall s>0, \quad   G (A, s)  \subset  G(A^\alpha, \alpha s)$$ The equality follows by exchanging the roles of $A$ and $A^\alpha$.
\end {proof} 

\subsection {Some operatorial inequalities}  
\begin{Prop}  Let $A$ be any positive self-adjoint operator on a Hilbert space H.  Then 

$$ \forall \beta \in [0, 1], \quad  (I+ A)^\beta \le I + A^\beta $$ 

$$ \forall \beta \in [0, 1],  \forall u \in D(A),\quad  |A^\beta u|^2 \le |u|^2 +|A u|^2  $$ 

\end{Prop} 
\begin{proof} 
The first inequality follows classically from the methods of operator calculus invoking the scalar inequality $$ \forall \beta \in [0, 1],\forall h>0, \quad  (I+ h)^\beta \le I + h^\beta $$  As for the second inequality we just write, assuming $\beta<1$,
 $$ |A^\beta u|^2 \le |Au|^{2\beta} |u|^{2(1-\beta )}\le [A|u|^{2\beta}]^{\frac{1}{\beta}} + [|u|^{2(1-\beta )}] ^{\frac{1}{(1-\beta)}} $$ as a consequence of interpolation and Young's inequality applied with the conjugate exponents $\frac{1}{\beta}$ and $\frac{1}{(1-\beta)}.$ 
 \end {proof}


 \end{document}